\documentclass[12pt,a4]{amsart}
\usepackage[a4paper, left=28mm, right=28mm, top=28mm, bottom=34mm]{geometry}
\usepackage[all]{xy}
\usepackage{amsmath}
\usepackage{amssymb}
\usepackage{amsthm}
\numberwithin{equation}{section}
\theoremstyle{definition}
\newtheorem{thm}{Theorem}[section]
\newtheorem{lem}[thm]{Lemma}
\newtheorem{cor}[thm]{Corollary}
\newtheorem{prop}[thm]{Proposition}

\theoremstyle{definition}

\newtheorem{rem}[thm]{Remark}

\newtheorem{defn}[thm]{Definition}
\newtheorem{prob}[thm]{Problem}

\def\F{{\mathbb F}}
\def\G{{\mathbb G}}
\def\Q{{\mathbb Q}}

\def\Z{{\mathbb Z}}

\def\OO{{\mathcal O}}
\def\PP{{\mathbb P}}

\def\Br{\mathop{\mathrm{Br}}\nolimits}

\def\Gal{\mathop{\mathrm{Gal}}\nolimits}

\def\inv{\mathop{\mathrm{inv}}\nolimits}

\def\Jac{\mathop{\mathrm{Jac}}\nolimits}
\def\Ker{\mathop{\mathrm{Ker}}\nolimits}
\def\id{\mathop{\mathrm{id}}\nolimits}

\def\GL{\mathop{\mathrm{GL}}\nolimits}

\def\SL{\mathop{\mathrm{SL}}\nolimits}

\def\Sp{\mathop{\mathrm{Sp}}\nolimits}

\def\Sp{\mathop{\mathrm{Sp}}\nolimits}

\def\Pic{\mathop{\mathrm{Pic}}\nolimits}

\def\Spec{\mathop{\rm Spec}}

\def\Tr{\mathop{\text{\rm Tr}}\nolimits}

\def\deg{\mathop{\text{\rm deg}}\nolimits}

\def\chara{\mathop{\mathrm{char}}}

\renewcommand{\labelenumi}{(\arabic{enumi})}
\begin{document}
\title[Local-global principle for symmetric determinantal representations]
{The local-global principle for symmetric determinantal representations of smooth plane curves}
\author{Yasuhiro Ishitsuka}
\address{Department of Mathematics, Faculty of Science, Kyoto University, Kyoto 606-8502, Japan}
\email{yasu-ishi@math.kyoto-u.ac.jp}
\author{Tetsushi Ito}
\address{Department of Mathematics, Faculty of Science, Kyoto University, Kyoto 606-8502, Japan}
\email{tetsushi@math.kyoto-u.ac.jp}

\date{\today}


\maketitle

\begin{abstract}
A smooth plane curve is said to admit a symmetric determinantal representation
if it can be defined by the determinant of a symmetric matrix with entries in
linear forms in three variables. We study the local-global principle for the existence of
symmetric determinantal representations of smooth plane curves over a global field
of characteristic different from two.
When the degree of the plane curve is less than or equal to three,
we relate the problem of finding
symmetric determinantal representations to more familiar Diophantine problems on
the Severi-Brauer varieties and mod 2 Galois representations,
and prove that the local-global principle holds for conics and cubics.
We also construct counterexamples to the local-global principle
for quartics using the results of Mumford, Harris, and Shioda on theta characteristics.
\end{abstract}


\section{Introduction}

Let $C \subset \PP^2_K$ be a smooth plane curve of degree $n \geq 1$ over a field $K$.
If there is a triple of symmetric matrices $(M_0,M_1,M_2)$ of size $n$
with entries in $K$ such that $C$ is defined by the equation
\[ \det\big(\,X_0 M_0 + X_1 M_1 + X_2 M_2\,\big) = 0, \]
we say $C$ admits a {\em symmetric determinantal representation} over $K$.
In this paper, we study the local-global principle for the existence of
symmetric determinantal representations of smooth plane curves
over a global field of characteristic different from two.
We prove that the local-global principle holds for smooth plane conics and smooth plane cubics.
We also construct counterexamples to the local-global principle for smooth plane quartics
using the results of Mumford, Harris, and Shioda on theta characteristics.

The following theorems are the main results of this paper.

\begin{thm}[see Theorem \ref{MainTheorem:LocalGlobal}]
\label{MainTheorem1}
Let $K$ be a global field of characteristic different from two,
and $C \subset \PP^2_K$ be a smooth plane conic or cubic
(i.e.,\ smooth plane curve of degree $n = 2$ or $3$) over $K$.
If $C$ admits a symmetric determinantal representation over
the completion $K_v$ for each place $v$ of $K$,
the smooth plane curve $C$ admits a symmetric determinantal representation over $K$.
\end{thm}

\begin{thm}[see Subsection \ref{Subsection:Counterexample}]
\label{MainTheorem2}
Let $K$ be a global field of characteristic different from two.
Let $C \subset \PP^2_K$ be a smooth plane quartic
(i.e.,\ smooth plane curve of degree 4)
such that the associated mod $2$ Galois representation
on the $2$-torsion points on the Jacobian variety $\Jac(C)$
\[
\rho_{C,2} \colon \Gal(K^{\mathrm{sep}}/K) \ \longrightarrow \ \Sp\big(\Jac(C)[2](K^{\mathrm{sep}})\,\big) \cong \Sp_{6}(\F_2)
\]
is surjective.
Then there is a finite extension $L/K$ such that
$C$ admits a symmetric determinantal representation over $L_w$ for each place $w$ of $L$,
and $C$ does not admit a symmetric determinantal representation over $L$.
\end{thm}

By the results of Harris, and Shioda,
there are smooth plane quartics with surjective associated mod 2 Galois representations
over $\Q$ and over $\F_p(T)$ for $p \notin \{ 2,3,5,7,11,29,1229 \}$
(\cite[p.\ 721]{HarrisGalois}, \cite[Theorem 7]{Shioda}).
Hence we have the following corollary.

\begin{cor}[see Subsection \ref{Subsection:Counterexample}]
\label{Corollary:MainTheorem2}
Let $p$ be an integer which is either zero or a prime number different from
$2,3,5,7,11,29,1229$.
Then there is a global field $K$ of characteristic $p$ and
a smooth plane quartic $C \subset \PP^2_K$ such that
$C$ admits a symmetric determinantal representation over $K_v$ for each place $v$ of $K$,
and $C$ does not admit a symmetric determinantal representation over $K$.
\end{cor}

\begin{rem}
In Theorem \ref{MainTheorem1},
we can replace ``for each place'' by ``for all but one places'' when $C$ is a conic.
Also, we can replace ``for each place'' by ``for all but finitely many places'' when $C$ is a cubic.
In Theorem \ref{MainTheorem2},
the assumption on the surjectivity of $\rho_{C,2}$ can be weakened
when $C$ has a $K$-rational point.
(See Theorem \ref{MainTheorem:LocalGlobal} and Subsection \ref{Subsection:Counterexample}.)
\end{rem}

\begin{rem}
\label{Remark:Characteristic2}
In this paper,
we only consider global fields of characteristic different from two.
The story is completely different in characteristic two.
The local-global principle {\em holds} 
for any smooth plane curves of any degree
over a global field of characteristic two
(\cite{IshitsukaItoCharacteristic2}).
\end{rem}

Concerning the local-global principle for the existence of
symmetric determinantal representations of smooth plane curves,
it seems interesting to study the following problems.

\begin{prob}
\begin{enumerate}
\item Are there counterexamples to the local-global principle
for smooth plane quartics over $\Q$ or $\F_p(T)$ for $p \neq 2$?
In other words, can we take
$K = \Q$ or $\F_p(T)$ for $p \neq 2$ in Corollary \ref{Corollary:MainTheorem2}?
\item Are there counterexamples to the local-global principle in degree $n \geq 5$?
\end{enumerate}
\end{prob}

We expect that the answers to the above problems are ``Yes,''
but we do not know how to prove them.
The reason is that we currently know very little about
the images of the associated mod 2 Galois representations of smooth plane curves of higher degree.
(See Remark \ref{Counterexample:Remark3}.)

Historically, finding symmetric determinantal representations of plane curves
is a classical problem in algebraic geometry,
which goes back to Hesse's work on plane cubics and quartics
(\cite{Hesse1}, \cite{Hesse2}, \cite[Chapter 4]{Dolgachev}).
If $K$ is algebraically closed of characteristic zero,
in 1902, Dixon proved the existence of symmetric determinantal representations
for generic plane curves of any degree (\cite{Dixon}).
Since then, this problem has been re-examined by many people
(\cite{Room}, \cite{Tyurin}, \cite{CookThomas}, \cite{Catanese},
\cite{Vinnikov1}, \cite{Vinnikov2}, \cite{BeauvillePrym}, \cite{Wall}, \cite{Meyer-Brandis}).
In 2000, Beauville systematically studied minimal resolutions of
coherent sheaves on the projective spaces,
and proved that all plane curves (including singular ones) admit symmetric determinantal representations
when $K$ is algebraically closed of characteristic zero.
The situation is quite different when $K$ is not algebraically closed.
In 1938, Edge studied the field of definition of
symmetric determinantal representations of the Fermat quartic and related curves
called {\em Edge's quartics} (\cite{Edge}).
In 2009, Wei Ho studied, among other things,
certain linear orbits of triples of matrices related to
symmetric determinantal representations of smooth plane curves
over a field of characteristic not dividing $3n(n-1)$ (\cite{Ho}).
The results of Beauville and Ho were generalized by the first author
to include the case of higher dimensional hypersurfaces over arbitrary fields (\cite{Ishitsuka}).
For symmetric determinantal representations of Fermat curves of prime degree
and the Klein quartic over $\Q$, see \cite{IshitsukaItoFermat}.

Let us give a sketch of the proof of our results.
The existence of a symmetric determinantal representation of a smooth plane curve $C$
of degree $n \leq 3$ is related to more familiar Diophantine problems.
When $n = 2$, it is related to the existence of a $K$-rational point on the conic
(Proposition \ref{DeterminantalRepresentationConic}).
When $n = 3$, it is related to the existence of a non-trivial $K$-rational $2$-torsion point on
the Jacobian variety $\Jac(C)$ (Proposition \ref{DeterminantalRepresentationCubic}).
We prove Theorem \ref{MainTheorem1} using these relations.
The proof of Theorem \ref{MainTheorem2} depends on a group theoretic
lemma on the action of subgroups of $\Sp_{2m}(\F_2)$
on quadratic forms over $\F_2$ (Lemma \ref{Lemma:GroupTheoreticLemma}).
We give a sufficient condition for a smooth plane curve of any degree
to violate the local-global principle in terms of the associated mod 2 Galois
representation (Proposition \ref{Proposition:CounterexampleSufficientCondition}).
For a smooth plane quartic whose associated mod 2 Galois representation is surjective,
it is not difficult to see that it satisfies the sufficient condition 
after taking a finite extension of the base field.

The outline of this paper is as follows:
In Section \ref{SectionCriterion},
we recall a relation between symmetric determinantal representations
and certain line bundles called {\em non-effective theta characteristics}.
In Section \ref{SectionPicardFunctor},
we recall the basic facts on the relative Picard functors and Picard schemes.
In Section \ref{SectionExamples},
we examine the case of $n \leq 3$ in some detail and prove a relation
between the existence of symmetric determinantal representations and other Diophantine problems.
Theorem \ref{MainTheorem1} is proved in Section \ref{SectionMainTheorem1}.
Finally, in Section \ref{SectionMainTheorem2},
after recalling a relation between theta characteristics and quadratic forms over $\F_2$,
we prove Theorem \ref{MainTheorem2}.

\subsection*{Notation}

For a field $K$,
an algebraic closure of it is denoted by $\overline{K}$,
and a separable closure of it is denoted by $K^{\mathrm{sep}}$.
A {\em global field} is a field isomorphic to a finite extension of $\Q$ or $\F_p(T)$,
where $p$ is a prime number, $\F_p$ is the finite field of order $p$, and $T$ is an indeterminate.
For a place $v$ of a global field $K$, the completion of $K$ at $v$ is denoted by $K_v$.
For a morphism of schemes $X \longrightarrow S$ and $S' \longrightarrow S$,
the base change $X \times_S S'$ is denoted by $X_{S'}$.
When $S = \Spec K$ and $S' = \Spec L$ for a field extension $L/K$,
the base change $X_{S'}$ is also denoted by $X \otimes_K L$ or $X_L$.

\section{Theta characteristics and symmetric determinantal representations}
\label{SectionCriterion}

We recall the definition of theta characteristics on proper smooth curves
and its relation to symmetric determinantal representations.
In this section, $K$ is a field of arbitrary characteristic.

\begin{defn}[\cite{MumfordTheta}]
\label{DefinitionTheta}
Let $C$ be a proper smooth geometrically connected curve over $K$.
\begin{enumerate}
\item A {\em theta characteristic} on $C$ is a line bundle $\mathcal{L}$ on $C$
satisfying $\mathcal{L} \otimes \mathcal{L} \cong \Omega^1_C$,
where $\Omega^1_C$ is the canonical sheaf on $C$.
\item A theta characteristic $\mathcal{L}$ on $C$ is {\em effective} (resp.\ {\em non-effective}) 
if $H^0(C,\mathcal{L}) \neq 0$ (resp.\ $H^0(C,\mathcal{L}) = 0$).
\item A theta characteristic $\mathcal{L}$ on $C$ is {\em even} (resp.\ {\em odd})
if $\dim_{K} H^0(C,\mathcal{L})$ is {\em even} (resp.\ {\em odd}).
\end{enumerate}
\end{defn}

\begin{thm}
\label{ExistenceCriterion}
Let $\iota \colon C \hookrightarrow \PP^2_K$ be a smooth plane curve  over $K$.
Then $C$ admits a symmetric determinantal representation over $K$ if and only if
there is a non-effective theta characteristic on $C$.
\end{thm}

\begin{proof}
This result is well known when the characteristic of $K$ is different from two
(\cite[Proposition 4.2]{BeauvilleDeterminantal}, \cite[Ch 4]{Dolgachev}, \cite{Ho}).
Although the proofs in \cite{BeauvilleDeterminantal}, \cite{Ho} are
written under the additional assumptions on the base field,
it is not difficult to modify the arguments to cover the case of
characteristic two (\cite[Remark 2.2]{BeauvilleDeterminantal}).
For a proof of this proposition which works over arbitrary fields,
see \cite{Ishitsuka}.
\end{proof}

\begin{rem}
\label{Remark:EvenThetaCharacteristicQuartic}
Let $C$ be a smooth plane quartic over $K$.
It is well known that all even theta characteristics on $C$ are non-effective.
This can be seen as follows.
Assume that there is an effective even theta characteristic $\mathcal{L}$ on $C$.
Then $\mathcal{L}$ is isomorphic to $\OO_C(D)$ for an effective divisor $D$ of degree $g(C) - 1 = 2$
and we have $\dim_K H^0(C,\OO_C(D)) \geq 2$.
There is a non-constant rational function $f$ with $\mathop{\mathrm{div}}(f) + D \geq 0$.
Then $f$ defines a morphism $f \colon C \longrightarrow \PP^1_K$ of degree $2$.
It contradicts to the well-known fact that smooth plane quartics are non-hyperelliptic
(\cite[IV, Exercise 3.2]{Hartshorne}).
\end{rem}

\section{Relative Picard functors and Picard schemes}
\label{SectionPicardFunctor}

We recall the basic definitions and properties of
relative Picard functors and Picard schemes
(\cite{BoschLuetkebohmertRaynaud}, \cite{KleimanPicard}).

For a morphism of schemes $f \colon X \longrightarrow S$,
the {\em relative Picard functor} $\Pic_{X/S}$ is the fppf sheaf associated
with the functor
\[ \big( \mathrm{Schemes}/S \big)^{\mathrm{op}} \longrightarrow \big( \mathrm{Sets} \big),
\quad S' \mapsto \Pic(X_{S'}), \]
where the Picard group $\Pic(X_{S'})$ is the group of isomorphism classes of
line bundles on $X_{S'} := X \times_S S'$
(\cite[\S 8.1, Definition 2]{BoschLuetkebohmertRaynaud}).
If $f \colon X \longrightarrow S$ is quasi-compact, quasi-separated,
and $f_{\ast}(\OO_X) = \OO_S$ holds,
we have the following exact sequence for each {\em flat} $S$-scheme $S'$:
\begin{equation}
\label{PicardFunctor:ExactSequence1}
\xymatrix{
0 \ar[r] & \Pic(S')
  \ar[r] & \Pic(X_{S'})
  \ar[r] & \Pic_{X/S}(S')
  \ar[r] & \Br(S')
  \ar[r] & \Br(X_{S'}),
}
\end{equation}
where
\[ \Br(S') := H^2_{\mathrm{fppf}}(S',\G_m), \qquad
   \Br(X_{S'}) := H^2_{\mathrm{fppf}}(X_{S'},\G_m) \]
are the cohomological Brauer groups calculated in the fppf topology
(\cite[\S 8.1, Proposition 4]{BoschLuetkebohmertRaynaud}).
Since the sheaf $\G_m$ is representable by a smooth scheme,
the cohomological Brauer groups $\Br(S'), \Br(X_{S'})$
can be calculated using \'etale topology
(\cite[Th\'eor\`eme 11.7]{GrothendieckBrauerIII}).
In particular, when $S' := \Spec K$ is the spectrum of a field $K$,
the cohomological Brauer group $\Br(\Spec K)$ is isomorphic to
the Brauer group of $K$ defined by Galois cohomology (\cite[Ch.\ X, \S 4]{SerreLocalFields}),
i.e.,
\[ \Br(\Spec K) \cong \Br(K)
   := H^2 \big( \Gal(K^{\mathrm{sep}}/K),\,(K^{\mathrm{sep}})^{\times} \big). \]
When $S = S' = \Spec K$ for a field $K$,
the Picard group $\Pic(\Spec K)$ is trivial, and the exact sequence
(\ref{PicardFunctor:ExactSequence1}) becomes
\begin{equation}
\label{PicardFunctor:ExactSequence2}
\xymatrix{
0 \ar[r] & \Pic(X)
  \ar[r] & \Pic_{X/K}(K)
  \ar[r] & \Br(K)
  \ar[r] & \Br(X).
}
\end{equation}
We have $\Pic(X) = \Pic_{X/K}(K)$
if $\Br(K)$ is trivial {\em or} $X$ has a $K$-rational point
(\cite[\S 8.1, Proposition 4]{BoschLuetkebohmertRaynaud}).

When $X$ is a proper scheme over a field $K$, the relative Picard functor
$\Pic_{X/K}$ is representable by a scheme which is locally of finite type over $K$
(\cite[\S 8.2, Theorem 3]{BoschLuetkebohmertRaynaud}).
The scheme representing the functor $\Pic_{X/K}$ is
called the {\em Picard scheme}.

When $X$ is a proper smooth geometrically connected curve over $K$,
the identity component $\Jac(X) := \Pic^0_{X/K} \subset \Pic_{X/K}$
is called the {\em Jacobian variety}.
It is an abelian variety whose dimension is equal to the genus $g(X)$ of $X$
(\cite[\S 9.2, Proposition 3]{BoschLuetkebohmertRaynaud}).

\begin{rem}
\label{RemarkPicardFunctor}
For a proper scheme $X$ over a field $K$ whose geometric fiber is
connected and reduced, the following are well known.
\begin{enumerate}

\item
\label{RemarkPicardFunctor:BaseChange}
We have $\Pic_{X_L/L} \cong (\Pic_{X/K})_L$ because
the formation of Picard schemes commutes with base change
(\cite[Exercise 9.4.4]{KleimanPicard}).

\item
\label{RemarkPicardFunctor:Injection}
The map $\Pic(X) \longrightarrow \Pic(X_L)$ is injective for any field extension $L/K$.

\item
\label{RemarkPicardFunctor:Galois}
For a Galois extension $L/K$,
we have $\Pic_{X/K}(L)^{\Gal(L/K)} = \Pic_{X/K}(K)$.
But $\Pic(X_L)^{\Gal(L/K)} = \Pic(X)$ does not hold in general.
\end{enumerate}
\end{rem}

\begin{prop}
\label{Proposition:RationalityLineBundle}
Let $X$ be a proper scheme over a field $K$ whose geometric fiber is
connected and reduced.
Let $L/K$ be a Galois extension which is not necessarily finite.
Let $\mathcal{L}$ be a line bundle on $X_L$
such that its class
$[\mathcal{L}] \in \Pic(X_L)$
is fixed by the action of $\Gal(L/K)$.
Assume that {\em at least one} of the following conditions is satisfied:
\begin{enumerate}
\renewcommand{\labelenumi}{(\alph{enumi})}
\renewcommand{\theenumi}{{\rm\alph{enumi}}}
\item \label{Proposition:RationalityLineBundle:TrivialBrauerGroup}
  $\Br(K)$ is trivial,

\item \label{Proposition:RationalityLineBundle:ExistenceRationalPoint}
$X$ has a $K$-rational point,

\item \label{Proposition:RationalityLineBundle:ExistenceRationalPointAfterExtension}
there is a finite extension $M/K$ and an integer $r \geq 1$ prime to $[M : K]$
such that
$X$ has an $M$-rational point and
$[\mathcal{L}^{\otimes r}]$ comes from a line bundle on $X$, or

\item \label{Proposition:RationalityLineBundle:LocalExistence}
$K$ is a global field and
there is a place $v_0$ of $K$ such that,
for any place $v \neq v_0$ of $K$ and a place $w$ of $L$ above $v$,
$[\mathcal{L}_{L_w}]$ comes from a line bundle on $X_{K_v}$.
\end{enumerate}
Then $[\mathcal{L}]$ comes from a line bundle on $X$.
\end{prop}

\begin{proof}
Since the image of $[\mathcal{L}]$ in $\Pic_{X/K}(L)$ is fixed by $\Gal(L/K)$,
it comes from an element $\alpha_K \in \Pic_{X/K}(K)$.
We shall show that the image of $\alpha_K$ in $\Br(K)$ is trivial.
(\ref{Proposition:RationalityLineBundle:TrivialBrauerGroup}) is obvious;
(\ref{Proposition:RationalityLineBundle:ExistenceRationalPoint}),
(\ref{Proposition:RationalityLineBundle:ExistenceRationalPointAfterExtension})
are standard;
(\ref{Proposition:RationalityLineBundle:LocalExistence})
follows from the injectivity of the map
\[
\Br(K)\ \longrightarrow \bigoplus_{v \neq v_0} \Br(K_v).
\]
in global class field theory
(\cite[\S 9, \S 10]{TateGCFT}, \cite[Theorem 8.1.17]{NeukirchSchmidtWingberg}).
\end{proof}

We give an application of Picard schemes to theta characteristics.

\begin{prop}
\label{Proposition:ThetaDescent}
Let $C$ be a proper smooth geometrically connected curve over a field $K$,
and $\mathcal{L}$ a line bundle on $C$.
Let $L/K$ be an extension of fields,
and $\mathcal{L}_L$ be the pullback of $\mathcal{L}$ to $C_L$.
\begin{enumerate}
\item \label{Proposition:ThetaDescent:1}
$\mathcal{L}$ is a theta characteristic on $C$ if and only if
$\mathcal{L}_L$ is a theta characteristic on $C_L$.

\item \label{Proposition:ThetaDescent:2}
$\mathcal{L}$ is effective (resp.\ non-effective) if and only if
$\mathcal{L}_L$ is effective (resp.\ non-effective).
\end{enumerate}
\end{prop}

\begin{proof}
(\ref{Proposition:ThetaDescent:1})
\ The element
$[\mathcal{L}_L]$ (resp.\ $[\Omega^1_{C_L}]$)
is the image of
$[\mathcal{L}]$ (resp.\ $[\Omega^1_C]$)
by the canonical map $\Pic(C) \longrightarrow \Pic(C_L)$.
Since this map is injective by Remark \ref{RemarkPicardFunctor} (\ref{RemarkPicardFunctor:Injection}),
we see that $2 [\mathcal{L}] \ = \ [\Omega^1_C]$ holds if and only if
$2 [\mathcal{L}_L] \ = \ [\Omega^1_{C_L}]$ holds.

\vspace{0.1in}

\noindent
(\ref{Proposition:ThetaDescent:2}) \ This follows from the equality
$\dim_K H^0(C,\mathcal{L}) = \dim_L H^0(C_L,\mathcal{L}_L)$.
\end{proof}

\section{Symmetric determinantal representations in degree $n \leq 3$}
\label{SectionExamples}

We examine the existence of symmetric determinantal representations of
smooth plane curves of degree $n \leq 3$ in some detail.
In this section, we fix a field $K$ of arbitrary characteristic.
Let $C \subset \PP^2_K$ be a smooth plane curve of degree $n \leq 3$.

\subsection{Lines ($n=1$)}
\label{Subsection:Lines}

Obviously, every line over $K$ admits a symmetric determinantal representation over $K$.
In order to illustrate the methods of this paper,
let us confirm it using Theorem \ref{ExistenceCriterion}.
The plane curve $C \subset \PP^2_K$ of degree $1$ is isomorphic to $\PP^1_K$ over $K$.
The Picard group of $\PP^1_K$ is isomorphic to $\Z$ generated by $[\OO_{\PP^1_K}(1)]$.
Since $\deg \Omega^1_{\PP^1_K} = -2$, the line bundle
$\OO_{\PP^1_K}(-1)$ is a unique theta characteristic on $\PP^1_K$, up to isomorphism.
It is non-effective because its degree is negative.
Hence $C$ admits a symmetric determinantal representation over $K$
by Theorem \ref{ExistenceCriterion}.

\subsection{Conics ($n=2$)}
\label{Subsection:Conics}

There is a natural map from the set of smooth plane conics over $K$
to the set of elements of $\Br(K)$ killed by $2$.
There are several ways to construct it.
The following method seems most suitable for our purposes:
since $C$ is smooth over $K$, the curve $C$ has a $K^{\mathrm{sep}}$-rational point
(\cite[\S 2.2, Corollary 13]{BoschLuetkebohmertRaynaud}).
Hence $C_{K^{\mathrm{sep}}}$ is isomorphic to $\PP^1_{K^{\mathrm{sep}}}$ over $K^{\mathrm{sep}}$.
The Picard group of $\PP^1_{K^{\mathrm{sep}}}$ is isomorphic to $\Z$ generated by
a line bundle of degree 1.
Since the action of $\Gal(K^{\mathrm{sep}}/K)$ on $\Pic_{C/K}(K^{\mathrm{sep}})$
does not change the degree of line bundles,
we have the following isomorphisms:
\[ \Pic_{C/K}(K) = \Pic_{C/K}(K^{\mathrm{sep}})^{\Gal(K^{\mathrm{sep}}/K)}
  = \Pic_{C/K}(K^{\mathrm{sep}})
  \overset{\mathrm{deg}}{\cong} \Z. \]
(See Remark \ref{RemarkPicardFunctor} (\ref{RemarkPicardFunctor:Galois})
for the first equality.)
There is a unique element $s \in \Pic_{C/K}(K)$ of degree 1.
We define $\alpha_C \in \Br(K)$
to be the image of $s$ in $\Br(K)$ by the following exact sequence
(cf.\ (\ref{PicardFunctor:ExactSequence2})):
\[
\xymatrix{
0 \ar[r] & \Pic(C)
  \ar[r] & \Pic_{C/K}(K)
  \ar[r] & \Br(K).
}
\]
Since $\deg \Omega^1_C = -2$,
we see that $\alpha_C$ is trivial if and only if $C$ has a line bundle of odd degree.
The element $\alpha_C \in \Br(K)$ is killed by $2$.

\begin{prop}
\label{DeterminantalRepresentationConic}
The following are equivalent:
\begin{enumerate}
\renewcommand{\labelenumi}{(\alph{enumi})}
\renewcommand{\theenumi}{{\rm\alph{enumi}}}
\item \label{DeterminantalRepresentationConic:Condition1}
  $C$ admits a symmetric determinantal representation over $K$.
\item \label{DeterminantalRepresentationConic:Condition2}
  $C$ is isomorphic to $\PP^1_K$ over $K$.
\item \label{DeterminantalRepresentationConic:Condition3}
  $C$ has a $K$-rational point.
\item \label{DeterminantalRepresentationConic:Condition4}
  $C$ has a line bundle of odd degree.
\item \label{DeterminantalRepresentationConic:Condition5}
  $\alpha_C \in \Br(K)$ is trivial.
\end{enumerate}
\end{prop}

\begin{proof}
The equivalence
$(\ref{DeterminantalRepresentationConic:Condition4})
\Leftrightarrow (\ref{DeterminantalRepresentationConic:Condition5})$
follows from the construction of $\alpha_C \in \Br(K)$ recalled as above.
The implications
$(\ref{DeterminantalRepresentationConic:Condition2})
\Rightarrow (\ref{DeterminantalRepresentationConic:Condition3})
\Rightarrow (\ref{DeterminantalRepresentationConic:Condition4})$
are obvious.
If there is a line bundle $\mathcal{L}$ on $C$ of odd degree,
there is a line bundle $\mathcal{L}'$ on $C$ of degree 1 because
$\deg \Omega^1_C = -2$.
The complete linear system of $\mathcal{L}'$ gives an isomorphism $C \cong \PP^1_K$
by Riemann-Roch (\cite[Proposition 7.4.1]{Liu}).
Hence
$(\ref{DeterminantalRepresentationConic:Condition4})
\Rightarrow (\ref{DeterminantalRepresentationConic:Condition2})$
follows.
Finally,
by Proposition \ref{Proposition:ThetaDescent},
a line bundle $\mathcal{L}$ on $C$ is a theta characteristic
if and only if $\deg \mathcal{L} = -1$.
It is non-effective because $\deg \mathcal{L}$ is negative.
By Theorem \ref{ExistenceCriterion},
the existence a line bundle $\mathcal{L}$ with $\deg \mathcal{L} = -1$
is equivalent to the existence of a symmetric determinantal representation of $C$ over $K$.
Hence
$(\ref{DeterminantalRepresentationConic:Condition4})
\Leftrightarrow (\ref{DeterminantalRepresentationConic:Condition1})$
follows.
\end{proof}

\subsection{Cubics ($n=3$)}
\label{Subsection:Cubics}

Since $C \subset \PP^2_K$ is a smooth plane cubic,
the Jacobian variety $\Jac(C)$ is an elliptic curve over $K$.
It is well known that $C$ is isomorphic to $\Jac(C)$
if and only if $C$ has a $K$-rational point.

\begin{prop}
\label{DeterminantalRepresentationCubic}
The following are equivalent:
\begin{enumerate}
\renewcommand{\labelenumi}{(\alph{enumi})}
\renewcommand{\theenumi}{{\rm\alph{enumi}}}
\item \label{DeterminantalRepresentationCubic:Condition1}
  $C$ admits a symmetric determinantal representation over $K$.
\item \label{DeterminantalRepresentationCubic:Condition2}
  $\Jac(C)$ has a non-trivial $K$-rational $2$-torsion point.
\end{enumerate}
\end{prop}

\begin{proof}
Since $\Omega^1_C$ is trivial,
by Theorem \ref{ExistenceCriterion} and Proposition \ref{Proposition:ThetaDescent},
$C$ admits a symmetric determinantal representation over $K$
if and only if
there is a non-trivial line bundle $\mathcal{L}$ on $C$
satisfying $\mathcal{L} \otimes \mathcal{L} \cong \OO_C$.

\vspace{0.1in}

\noindent
$(\ref{DeterminantalRepresentationCubic:Condition1})
\Rightarrow (\ref{DeterminantalRepresentationCubic:Condition2})$
\ If $\mathcal{L}$ is a non-trivial line bundle
satisfying $\mathcal{L} \otimes \mathcal{L} \cong \OO_C$,
its class $[\mathcal{L}] \in \Jac(C)(K)$ is a non-trivial $K$-rational $2$-torsion point.

\vspace{0.1in}

\noindent
$(\ref{DeterminantalRepresentationCubic:Condition2})
\Rightarrow (\ref{DeterminantalRepresentationCubic:Condition1})$
\ Let $\alpha \in \Jac(C)(K)$
be a non-trivial $K$-rational $2$-torsion point.
There is a finite extension $M/K$ of odd degree with $C(M) \neq \emptyset$
because $C$ has odd degree.
Hence $\alpha$ comes from a line bundle $\mathcal{L}_{\alpha}$ on $C$
by Proposition \ref{Proposition:RationalityLineBundle}
(\ref{Proposition:RationalityLineBundle:ExistenceRationalPointAfterExtension})
for $r = 2$.
\end{proof}

\section{The local-global principle for conics and cubics}
\label{SectionMainTheorem1}

The following theorem is slightly more general than Theorem \ref{MainTheorem1}.

\begin{thm}
\label{MainTheorem:LocalGlobal}
Let $K$ be a global field of characteristic different from two,
and $C \subset \PP^2_K$ be a smooth plane curve of degree $2$ or $3$ over $K$.
\begin{enumerate}
\item \label{MainTheorem:LocalGlobal:Conic}
Assume that $C$ has degree $2$.
If there is a place $v_0$ of $K$ such that
$C$ admits symmetric determinantal representations over $K_v$ for {\em all places}
$v \neq v_0$ of $K$,
the smooth plane curve $C$ admits a symmetric determinantal representation over $K$.

\item \label{MainTheorem:LocalGlobal:Cubic}
Assume that $C$ has degree $3$.
If $C$ admits symmetric determinantal representations over
$K_v$ for {\em all but finitely many places} $v$ of $K$,
the smooth plane curve $C$ admits a symmetric determinantal representation over $K$.
\end{enumerate}
\end{thm}

\begin{proof}
\noindent
(\ref{MainTheorem:LocalGlobal:Conic})
\ Let $\alpha_C \in \Br(K)$ be the element associated with $C$
in Subsection \ref{Subsection:Conics}.
Since the image of $\alpha_C$ in $\Br(K_v)$ is trivial
for each $v \neq v_0$,
we see that $\alpha_C$ is trivial by the structure of the Brauer group
of $K$ (\cite[\S 9, \S 10]{TateGCFT}, \cite[Theorem 8.1.17]{NeukirchSchmidtWingberg}).
Hence $C$ admits a symmetric determinantal representation over $K$
by Proposition \ref{DeterminantalRepresentationConic}.

\vspace{0.1in}

\noindent
(\ref{MainTheorem:LocalGlobal:Cubic})
\ The non-trivial 2-torsion points on $\Jac(C)$
are defined by a cubic polynomial $f(X) \in K[X]$.
The assertion follows from the well known fact that
$f(X) = 0$ has a solution in $K$
if it has a solution in $K_v$ for all but finitely many $v$
by Chebotarev's density theorem (\cite[I.2.2]{SerreAbelianl-adic}).
\end{proof}

\begin{rem}
Theorem \ref{MainTheorem:LocalGlobal}
(\ref{MainTheorem:LocalGlobal:Conic}) is optimal in the following sense:
for any global field $K$ and any finite set $S$ of places of $K$ of cardinality $\geq 2$,
there is a smooth plane conic $C \subset \PP^2_K$
such that $C$ admits a symmetric determinantal representation over $K_v$ for each $v \notin S$,
and $C$ does not admit a symmetric determinantal representation over $K$.
To see this,
let $\alpha \in \Br(K)$ be a non-trivial element  killed by $2$
satisfying $\inv_v(\alpha) = 0$ for all $v \notin S$
(\cite[\S 9, \S 10]{TateGCFT}, \cite[Theorem 8.1.17]{NeukirchSchmidtWingberg}),
and consider the conic (the Severi-Brauer variety) associated with it
(\cite[Ch X, \S 6]{SerreLocalFields}).
\end{rem}

\begin{rem}
The following argument seems instructive to understand how to construct
counterexamples to the local-global principle for quartics
in Section \ref{SectionMainTheorem2}
(see also Remark \ref{Lemma:GroupTheoreticLemma:Falsem=1}).
The proof of Theorem \ref{MainTheorem:LocalGlobal}
(\ref{MainTheorem:LocalGlobal:Cubic})
can be rephrased in terms of mod 2 Galois representations.
Choose an $\F_2$-basis on $\Jac(C)[2](K^{\mathrm{sep}})$,
and consider the mod $2$ Galois representation on it:
\[ \rho_{C,2} \colon \Gal(K^{\mathrm{sep}}/K) \ \longrightarrow \ \GL_2(\F_2). \]
Let $v$ be a finite place of $K$ such that
$C$ admits a symmetric determinantal representation over $K_v$
and $\Jac(C)$ has good reduction at $v$.
By Proposition \ref{DeterminantalRepresentationCubic},
the image of the geometric Frobenius element $\rho_{C,2}(\mathrm{Frob}_v)$
has a non-zero fixed vector.
Let
\[ G := \rho_{C,2} \big( \Gal(K^{\mathrm{sep}}/K) \big)
 \subset \GL_2(\F_2) \]
be the image of $\rho_{C,2}$,
which is generated by $\rho_{C,2}(\mathrm{Frob}_v)$ for all but finitely many $v$
by Chebotarev's density theorem (\cite[I.2.2]{SerreAbelianl-adic}).
By an easy group theoretic lemma (see Lemma \ref{LemmaGL(2,2)} below),
there is a non-zero vector fixed by all elements of $G$.
Hence $\Jac(C)$ has a non-trivial $K$-rational $2$-torsion point,
and $C$ admits a symmetric determinantal representation over $K$
by Proposition \ref{DeterminantalRepresentationCubic}.
\end{rem}

The proof of the following lemma is easy and omitted.

\begin{lem}
\label{LemmaGL(2,2)}
Let $G \subset \GL_2(\F_2)$ be a subgroup such that,
for each element $g \in G$,
the action of $g$ on $\F_2^{\oplus 2}$ has a non-zero fixed vector.
Then there is a non-zero vector $v \in \F_2^{\oplus 2}$ fixed by all elements of $G$.
\end{lem}

\section{Counterexamples to the local-global principle for quartics}
\label{SectionMainTheorem2}

We recall the basic results on quadratic forms over $\F_2$
in Subsection \ref{Subsection:QuadraticFormsF2}.
Then we recall the Mumford's results on a relation between
theta characteristics and quadratic forms over $\F_2$
in Subsection \ref{Subsection:ThetaCharacteristicsQuadraticFormsF2}.
In Subsection \ref{Subsection:GroupTheoreticLemma},
we prove a group theoretic lemma on the action of a subgroup of
$\Sp_{2m}(\F_2)$ on quadratic forms over $\F_2$ (Lemma \ref{Lemma:GroupTheoreticLemma}).
In Subsection \ref{Subsection:Criterion},
we give a sufficient condition for a smooth plane curve over a global field 
to violate the local-global principle
in terms of the associated mod 2 Galois representation
(Proposition \ref{Proposition:CounterexampleSufficientCondition}).
In Subsection \ref{Subsection:Counterexample},
we prove Theorem \ref{MainTheorem2} and
construct counterexamples to the local-global principle for quartics.

\subsection{Quadratic forms over $\F_2$}
\label{Subsection:QuadraticFormsF2}

We recall the basic results on quadratic forms over the finite field $\F_2$
of order two.
(For more details on the group theoretic properties of
the action of $\Sp_{2m}(\F_2)$ on quadratic forms,
see \cite{Dye}, \cite[\S 7.7]{DixonMortimer}, \cite{GrossHarris}.)

Let $\F_2^{\oplus 2m}$ be the $2m$-dimensional vector space over $\F_2$.
Let
\[ \{\ e_1,\ldots,e_m,\ f_1,\ldots,f_m \ \} \]
be the standard basis of $\F_2^{\oplus 2m}$,
and define the alternating bilinear form $\langle,\rangle$ on $\F_2^{\oplus 2m}$ by
\[ \langle e_i,e_j \rangle = 0,
\qquad \langle f_i,f_j \rangle = 0,
\qquad \langle e_i,f_j \rangle = \langle f_j,e_i \rangle
= \begin{cases} 1 & i=j, \\ 0 & i \neq j. \end{cases}.
\]

The symplectic group $\Sp_{2m}(\F_2)$ is defined to be
the group of $\F_2$-linear automorphisms of $\F_2^{\oplus 2m}$
preserving the alternating form $\langle,\rangle$.
A {\em quadratic form} on $\F_2^{\oplus 2m}$ with polar form $\langle,\rangle$
is a map
$Q \colon \F_2^{\oplus 2m} \longrightarrow \F_2$
satisfying
\[ Q(x+y) - Q(x) -Q(y) = \langle x,y \rangle \]
for all $x,y \in \F_2^{\oplus 2m}$.

There are $2^{2m}$ quadratic forms on $\F_2^{\oplus 2m}$
with polar form $\langle,\rangle$.
The symplectic group $\Sp_{2m}(\F_2)$ acts on the set of quadratic forms
with polar form $\langle,\rangle$ by
\[ (g \cdot Q)(x) := Q(g^{-1} x) \]
for $g \in \Sp_{2m}(\F_2),\ x \in \F_2^{\oplus 2m}$.
This action has two orbits $\Omega^{+}, \Omega^{-}$ of
size $2^{m-1}(2^{m} + 1), 2^{m-1}(2^{m} - 1)$, respectively.
These orbits are distinguished by the {\em Arf invariant}.
There are several equivalent definitions of the Arf invariant.
One definition of the Arf invariant of a quadratic form $Q$ is
\[ \sum_{i=1}^{m} Q(e_i) Q(f_i) \in \F_2, \]
which is shown to be independent of the choice of the symplectic basis.
Another impressive definition is this:
the Arf invariant of $Q$ is $a$ ($a \in \{ 0, 1 \}$) if and only if
the number of elements $x \in \F_2^{\oplus 2m}$ with $Q(x) = a$
is equal to $2^{m-1}(2^{m} + 1)$ (\cite[Corollary 1.12]{GrossHarris}).

\subsection{Theta characteristics and quadratic forms over $\F_2$}
\label{Subsection:ThetaCharacteristicsQuadraticFormsF2}

We recall a relation between theta characteristics and quadratic forms over $\F_2$
due to Mumford (\cite{MumfordTheta}).

Let $C$ be a proper smooth geometrically connected curve of genus $g$ over a field $K$.
Assume that the characteristic of $K$ is {\em different from two}.
The Jacobian variety $\Jac(C)$, which is
the identity component of the Picard scheme $\Pic_{C/K}$,
is an abelian variety of dimension $g$
(cf.\ Section \ref{SectionPicardFunctor}).
The multiplication-by-2 isogeny
\[ [2] \colon \Pic_{C/K} \longrightarrow \Pic_{C/K} \]
is \'etale because $2$ is invertible in $K$.
The group scheme $\Jac(C)[2]$ is finite and \'etale over $K$ of order $2^{2g}$.
All $\overline{K}$-rational points on $\Jac(C)[2]$
are defined over $K^{\mathrm{sep}}$.
We see that
\[ \Jac(C)[2](K^{\mathrm{sep}}) = \Jac(C)[2](\overline{K}) \]
is an $\F_2$-vector space of dimension $2g$.
We have the {\em Weil pairing}
\[
e_2 \colon \mathrm{Jac}(C)[2](K^{\mathrm{sep}}) \times \mathrm{Jac}(C)[2](K^{\mathrm{sep}}) \longrightarrow \{ \pm 1 \} \cong \F_2,
\]
which is an alternating bilinear form over $\F_2$.
Here the multiplicative group $\{ \pm 1 \}$
is isomorphic to the additive group of $\F_2$.
Mumford proved that, for a theta characteristic $\mathcal{L}$ on $C_{K^{\mathrm{sep}}}$,
the map
\[ Q_{\mathcal{L}} \colon \mathrm{Jac}(C)[2](K^{\mathrm{sep}}) \longrightarrow \F_2 \]
defined by
\[ [\mathcal{M}] \mapsto
  \big( \dim_{K^{\mathrm{sep}}} H^0(C_{K^{\mathrm{sep}}},\mathcal{L} \otimes \mathcal{M})
    + \dim_{K^{\mathrm{sep}}} H^0(C_{K^{\mathrm{sep}}},\mathcal{L}) \big) \pmod{2}
\]
is a quadratic form with polar form $e_2$.
Any quadratic form on $\mathrm{Jac}(C)[2](K^{\mathrm{sep}})$ with polar form $e_2$
can be written as $Q_{\mathcal{L}}$
for a theta characteristic $\mathcal{L}$ on $C_{K^{\mathrm{sep}}}$, and
the isomorphism class of $\mathcal{L}$ is uniquely determined by its associated quadratic form.
The Arf invariant of $Q_{\mathcal{L}}$ is 0 (resp.\ 1)
if and only if $\mathcal{L}$ is even (resp.\ odd) (see Definition \ref{DefinitionTheta}).

\begin{rem}
\label{Remark:MumfordSeparableClosure}
A cautious reader might note that
Mumford worked over {\em an algebraic closure $\overline{K}$} of $K$ rather than
a separable closure $K^{\mathrm{sep}}$ of $K$ (\cite{MumfordTheta}).
It is easy to see that his results are valid over $K^{\mathrm{sep}}$ as well.
To see this, it is enough to note the following:
\begin{enumerate}
\item $\Jac(C)[2](K^{\mathrm{sep}}) = \Jac(C)[2](\overline{K})$ because $2$ is invertible in $K$.
\item For a theta characteristic $\mathcal{L}$ on $C_{K^{\mathrm{sep}}}$,
we have $Q_{\mathcal{L}} = Q_{\mathcal{L}_{\overline{K}}}$.

\item Every $\overline{K}$-rational (resp.\ $K^{\mathrm{sep}}$-rational)
point on the Picard scheme $\Pic_{C/K}$ comes from a line bundle on
$C_{\overline{K}}$ (resp.\ $C_{K^{\mathrm{sep}}}$)
because the Brauer group $\Br(\overline{K})$ (resp.\ $\Br(K^{\mathrm{sep}})$) is trivial
(\cite[\S 8.1, Proposition 4]{BoschLuetkebohmertRaynaud}).

\item The set of isomorphism classes of theta characteristics on $C_{\overline{K}}$
(resp.\ $C_{K^{\mathrm{sep}}}$) is identified with
$\big( [2]^{-1}([\Omega^1_C]) \big)(\overline{K})$
(resp.\ $\big( [2]^{-1}([\Omega^1_C]) \big)(K^{\mathrm{sep}})$).
Since $[2]$ is an \'etale isogeny, these two sets are equal.
Hence the set of isomorphism classes of theta characteristics on $C_{\overline{K}}$
and on $C_{K^{\mathrm{sep}}}$ are canonically identified.
\end{enumerate}
\end{rem}

From the above remarks, we have the following proposition.

\begin{prop}
\label{Proposition:MumfordBijection}
For each $a \in \{ 0,1 \}$, the association
$\mathcal{L} \mapsto Q_{\mathcal{L}}$
gives a bijection between the following sets:
\begin{itemize}
\item The set of quadratic forms on $\mathrm{Jac}(C)[2](K^{\mathrm{sep}})$
of Arf invariant $a$ whose polar form is the Weil pairing $e_2$.
\item The set of isomorphism classes of 
theta characteristics $\mathcal{L}$ on $C_{K^{\mathrm{sep}}}$ satisfying
$\dim_{K^{\mathrm{sep}}} H^0(C_{K^{\mathrm{sep}}},\mathcal{L}) \equiv a \pmod{2}$.
\end{itemize}
The bijection is equivariant with respect to the action
of $\Gal(K^{\mathrm{sep}}/K)$.
\end{prop}

From Proposition \ref{Proposition:RationalityLineBundle},
we have the following corollary.

\begin{cor}
\label{Corollary:MumfordBijectionGaloisEquivariance}
Let $Q$ be a quadratic form on $\Jac(C)[2](K^{\mathrm{sep}})$
whose polar form is the Weil pairing $e_2$.
If $Q$ is fixed by the action of $\Gal(K^{\mathrm{sep}}/K)$,
there is a theta characteristic $\mathcal{L}$ on $C_{K^{\mathrm{sep}}}$
with $Q = Q_{\mathcal{L}}$
by Proposition \ref{Proposition:MumfordBijection}.
Then the line bundle $\mathcal{L}$ on $C_{K^{\mathrm{sep}}}$
comes from a theta characteristic on $C$
if {\em at least one} of the following conditions is satisfied:
\begin{enumerate}
\renewcommand{\labelenumi}{(\alph{enumi})}
\renewcommand{\theenumi}{{\rm\alph{enumi}}}
\item \label{Corollary:MumfordBijectionGaloisEquivariance:BrauerTrivial}
  $\Br(K)$ is trivial,

\item \label{Corollary:MumfordBijectionGaloisEquivariance:ExistenceRationalPoint}
  $C$ has a $K$-rational point,

\item \label{Corollary:MumfordBijectionGaloisEquivariance:ExistenceRationalPointOddDegree}
  there is a finite extension $M/K$ of {\em odd} degree such that $C$ has an $M$-rational point, or

\item \label{Corollary:MumfordBijectionGaloisEquivariance:LocalExistence}
  $K$ is a global field, and there is a place $v_0$ of $K$ such that
  $\mathcal{L}_{K_v^{\mathrm{sep}}}$ comes from a theta characteristic on $C_{K_v}$
  for any place $v \neq v_0$ of $K$.
\end{enumerate}
\end{cor}

\subsection{Group theoretic lemmas}
\label{Subsection:GroupTheoreticLemma}

We prove a group theoretic lemma
on the action of subgroups of $\Sp_{2m}(\F_2)$ on quadratic forms over $\F_2$
(Lemma \ref{Lemma:GroupTheoreticLemma}).
The case of $m=3$ of Lemma \ref{Lemma:GroupTheoreticLemma}
will be used to construct counterexamples to the local-global principle
for quartics (cf.\ Subsection \ref{Subsection:Counterexample}).

In the following,
we use the same notation as in Subsection \ref{Subsection:QuadraticFormsF2}.

\begin{lem}
\label{Sublemma:GroupTheoreticLemma}
Fix an integer $m \geq 1$,
and a quadratic form $Q$ on $\F_2^{\oplus 2m}$ of Arf invariant 1
whose polar form is the standard alternating bilinear form $\langle,\rangle$.
Let
$\mathrm{O}(Q) \subset \Sp_{2m}(\F_2)$
be the orthogonal group associated with $Q$,
which is the group of $\F_2$-linear automorphisms of $\F_2^{\oplus 2m}$ preserving $Q$.
We denote the identity element by $e \in \mathrm{O}(Q)$.
Then there are elements
$\sigma, \tau \in \mathrm{O}(Q)$
satisfying all of the following conditions:
\begin{enumerate}
\renewcommand{\labelenumi}{(\alph{enumi})}
\renewcommand{\theenumi}{{\rm\alph{enumi}}}
\item \label{Sublemma:GroupTheoreticLemma:Condition1}
$\sigma \neq e$,

\item \label{Sublemma:GroupTheoreticLemma:Condition2}
$\tau^2 = e$,

\item \label{Sublemma:GroupTheoreticLemma:Condition3}
$\tau \sigma \tau = \sigma^{-1}$,

\item \label{Sublemma:GroupTheoreticLemma:Condition4}
$\sigma^i$ has no non-zero fixed vector in $\F_2^{\oplus 2m}$
for any $i$ with $\sigma^i \neq e$, and

\item \label{Sublemma:GroupTheoreticLemma:Condition5}
$\tau \sigma^i$ has a non-zero fixed vector
$x \in \F_2^{\oplus 2m}$ with $Q(x) = 1$ for any $i$.
\end{enumerate}
Note that,  in the condition (\ref{Sublemma:GroupTheoreticLemma:Condition5}),
the fixed vector $x \in \F_2^{\oplus 2m}$ of $\tau \sigma^i$ may depend on $i$.
\end{lem}

\begin{rem}
\label{Sublemma:Remark:GroupTheoreticLemma}
The subgroup of $\mathrm{O}(Q)$ generated by $\sigma, \tau$
is isomorphic to the dihedral group of order $2 n(\sigma)$,
where $n(\sigma)$ is the order of $\sigma$.
\end{rem}

\begin{proof}
Since $\Sp_{2m}(\F_2)$ acts transitively on the set of quadratic forms with Arf invariant 1,
it is enough to prove the assertion for a particular quadratic form with Arf invariant 1.

Let $\F_{2^{2m}}$ be the finite field of order $2^{2m}$.
Let us consider $\F_{2^{2m}}$ as an $\F_2$-vector space of dimension $2m$.
We shall construct an alternating bilinear form $\langle,\rangle$ and
a quadratic form $Q$ with Arf invariant 1 as follows.
In order to shorten the notation, we put $F(k) := \F_{2^{k}}$.
For $x \in F(2m)$, the conjugate of $x$ over $F(m)$ is denoted by $\overline{x}$.
For $x,y \in F(2m)$, we define $\langle x,y \rangle$ and $Q(x)$ by
\begin{align*}
  \langle x,y \rangle &:= \Tr_{F(2m)/F(1)} (x \overline{y}), \\
  \quad Q(x) &:= \Tr_{F(m)/F(1)} \big( \,\mathrm{N}_{F(2m)/F(m)}(x) \big).
\end{align*}
It is a routine exercise to check that $\langle,\rangle$ is an alternating bilinear form on $F(2m)$,
and $Q$ is a quadratic form on $F(2m)$ with polar form $\langle,\rangle$.

We shall show that the Arf invariant of $Q$ is 1.
We count the number of elements $x \in F(2m)$ with $Q(x) = 1$ as follows.
Since $F(m)/F(1)$ is a separable extension,
the trace map
\[ \Tr_{F(m)/F(1)} \colon F(m) \longrightarrow F(1) \]
is surjective.
The number of elements $t \in F(m)$ with $\Tr_{F(m)/F(1)}(t) = 1$ is $2^{m-1}$.
The norm map
\[ \mathrm{N}_{F(2m)/F(m)} \colon F(2m)^{\times} \longrightarrow F(m)^{\times} \]
is surjective (\cite[Ch.\ X, \S 7]{SerreLocalFields}).
Hence the number of elements $x \in F(2m)$ with $Q(x) = 1$
is equal to
\[ 2^{m-1} \cdot [F(2m)^{\times} : F(m)^{\times}] = 2^{m-1}(2^m + 1), \]
and the Arf invariant of $Q$ is 1.

Let $\mathrm{O}(Q)$ be the group of $\F_2$-linear automorphisms of $F(2m)$
preserving $Q$.
We shall construct two elements in $\mathrm{O}(Q)$
satisfying all of the conditions of this lemma.
Since $F(2m)^{\times}$ is a cyclic group,
the kernel of the norm map
\[ \mathrm{N}_{F(2m)/F(m)} \colon F(2m)^{\times} \longrightarrow F(m)^{\times} \]
is also cyclic. We take a generator $s \in F(2m)^{\times}$
of the kernel of $\mathrm{N}_{F(2m)/F(m)}$.
We define $g \in \mathrm{O}(Q)$ by $g(x) := s x$.
We define $h \in \mathrm{O}(Q)$ by $h(x) := \overline{x}$.

We shall prove that the elements $g,h \in \mathrm{O}(Q)$ satisfy
the required conditions for $\sigma, \tau$.
The conditions
(\ref{Sublemma:GroupTheoreticLemma:Condition1}),
(\ref{Sublemma:GroupTheoreticLemma:Condition2})
are obvious because $s \neq 1$ and $\overline{\overline{x}} = x$ for $x \in F(2m)$.
The condition
(\ref{Sublemma:GroupTheoreticLemma:Condition3})
is satisfied because we have
\[ (h \circ g \circ h)(x) = \overline{s \overline{x}}
    = s^{-1} x
    = g^{-1}(x). \]
If $g^i$ is not the identity element, we see that
$s^i \neq 1$ and the map
\[ x \mapsto g^i(x) = s^i x \]
has no non-zero fixed vector.
Hence the condition
(\ref{Sublemma:GroupTheoreticLemma:Condition4})
is satisfied.
Since $s^{-1} = \overline{s}$,
we have
\[ (h \circ g^i)(x) = x \iff s^{-i} \overline{x} = x \]
for $x \in F(2m)$.
Since $F(2m)/F(m)$ is a quadratic Galois extension and
\[ \mathrm{N}_{F(2m)/F(m)}(s^{-i}) = 1, \]
there is an element $y \in F(2m)^{\times}$ satisfying
$s^{-i} \overline{y} = y$
by Hilbert's Theorem 90 (\cite[Ch.\ X, \S 1]{SerreLocalFields}).
The condition ``$s^{-i} \overline{y} = y$'' is satisfied if we replace $y$ by $ty$
for $t \in F(m)^{\times}$.
The map
\[ F(m)^{\times} \longrightarrow F(m)^{\times},\ t \mapsto t^2 \]
is surjective because $F(m)^{\times}$ is a finite abelian group of odd order.
For an element $t \in F(m)^{\times}$,
we have
\begin{align*}
  Q(ty) &:= \Tr_{F(m)/F(1)} \big( \,\mathrm{N}_{F(2m)/F(m)}(ty) \big) \\
        &:= \Tr_{F(m)/F(1)} \big( \, t^2 \mathrm{N}_{F(2m)/F(m)}(y) \big).
\end{align*}
The bilinear form
\[
F(m) \times F(m) \longrightarrow F(1), \qquad
(u,v) \mapsto \Tr_{F(m)/F(1)} \big( uv  \big)
\]
is non-degenerate because $F(m)/F(1)$ is a separable extension.
Therefore,
after replacing $y$ by $ty$ for some $t \in F(m)^{\times}$,
we have $Q(y) = 1$.
The condition (\ref{Sublemma:GroupTheoreticLemma:Condition5}) is satisfied.
\end{proof}

\begin{lem}
\label{Lemma:GroupTheoreticLemma}
For $m \geq 3$, there is a subgroup
$G \subset \Sp_{2m}(\F_2)$
satisfying both of the following conditions:
\begin{enumerate}
\renewcommand{\labelenumi}{(\alph{enumi})}
\renewcommand{\theenumi}{{\rm\alph{enumi}}}
\item \label{Lemma:GroupTheoreticLemma:1}
there does not exist a quadratic form of Arf invariant $0$
with polar form $\langle,\rangle$ fixed by all elements of $G$, and

\item \label{Lemma:GroupTheoreticLemma:2}
for each $g \in G$, there is a quadratic form
of Arf invariant $0$ with polar form $\langle,\rangle$ fixed by $g$.
\end{enumerate}
\end{lem}

\begin{proof}
Let us decompose $\F_2^{\oplus 2m}$ as
\[ \F_2^{\oplus 2m} 
    \cong \F_2^{\oplus 2m-4} \oplus \F_2^{\oplus 2} \oplus \F_2^{\oplus 2}
    \cong V_1 \oplus V_2 \oplus V_3. \]
We put the standard alternating bilinear form on each $V_i$.

We shall define a quadratic form on each $V_i$ as follows.
Let $Q_1$ be a quadratic form on $V_1$ of Arf invariant 1.
Let $\{ e_2,f_2 \}$ (resp.\ $\{ e_3,f_3 \}$) be the standard symplectic basis
of $V_2$ (resp.\ $V_3$).
For $i = 2,3$, we define a quadratic form $Q_i$ on $V_i$ by
\[ Q_i(ae_i + bf_i) := a^2 + ab + b^2 \]
for $a,b \in \F_2$.
The Arf invariant of $Q_i$ is 1.
Hence the Arf invariant of the direct sum
\[ Q := Q_1 \oplus Q_2 \oplus Q_3 \]
is 1.
We denote the orthogonal group associated with the quadratic form
$Q_1$ (resp.\ $Q_2 \oplus Q_3, Q$)
by $\mathrm{O}(Q_1)$ (resp.\ $\mathrm{O}(Q_2 \oplus Q_3), \mathrm{O}(Q)$).
The product $\mathrm{O}(Q_1) \times \mathrm{O}(Q_2 \oplus Q_3)$
is a subgroup of $\mathrm{O}(Q)$.

Now we shall apply Lemma \ref{Sublemma:GroupTheoreticLemma} to $Q_1$.
We get two elements $\sigma, \tau \in \mathrm{O}(Q_1)$.
We define $\eta \in \mathrm{O}(Q_2 \oplus Q_3)$ by
\[ \eta(e_2) := e_3, \quad \eta(f_2) := f_3, \quad \eta(e_3) := e_2, \quad \eta(f_3) := f_2. \]
Let
$G := \langle (\sigma,\id), (\tau,\eta) \rangle$
be the subgroup of $\mathrm{O}(Q)$
generated by $(\sigma,\id)$ and $(\tau,\eta)$.

We shall show that $G$ satisfies the conditions
(\ref{Lemma:GroupTheoreticLemma:1}),
(\ref{Lemma:GroupTheoreticLemma:2})
of this lemma.
The following are well known (\cite[Lemma 1]{Dye}):
\begin{itemize}
\item For each vector $v \in \F_2^{\oplus 2m}$, the map
\[ Q_v \colon \F_2^{\oplus 2m} \longrightarrow \F_2,
   \quad x \mapsto Q_v(x) := Q(x) + \langle x,v \rangle \]
is a quadratic form on $\F_2^{\oplus 2m}$
with polar form $\langle,\rangle$.

\item Every quadratic form on $\F_2^{\oplus 2m}$ with polar form $\langle,\rangle$ can be written
as $Q_v$ for a unique vector $v \in \F_2^{\oplus 2m}$.

\item The Arf invariant of $Q_v$ is $0$ (resp.\ $1$)
if and only if $Q(v) = 1$ (resp.\ $0$).

\item For $g \in \mathrm{O}(Q)$, we see that
$g \cdot Q_v = Q_v$ if and only if $gv = v$
because
\[ (g \cdot Q_v)(x) = Q_v(g^{-1} x)
  = Q(g^{-1} x) + \langle g^{-1} x,\, v \rangle
   = Q(x) + \langle x,\, g v \rangle. \]
\end{itemize}

In order to check the condition (\ref{Lemma:GroupTheoreticLemma:1}),
it is enough to prove that there does not exist a vector
$v \in \F_2^{\oplus 2m}$ with $Q(v) = 1$ which is fixed by all elements of $G$.
Since $\sigma$ has no non-zero fixed vector in $V_1$,
a vector $v \in \F_2^{\oplus 2m}$ fixed by all elements of $G$ is
necessarily of the form
\[ v = a(e_2 + e_3) + b (f_2 + f_3) \]
for some $a,b \in \F_2$.
Since
\[ (Q_2 \oplus Q_3)(v) = (a^2 + ab + b^2) + (a^2 + ab + b^2) = 0, \]
there does not exist a vector $v \in \F_2^{\oplus 2m}$ with $Q(v) = 1$ fixed by
all elements of $G$.

We shall check the condition (\ref{Lemma:GroupTheoreticLemma:2}).
Any element $g \in G$ is either of the form $(\sigma^i,\id)$ or $(\tau \sigma^i,\eta)$.
If $g = (\sigma^i,\id)$,
all vectors in $V_2 \oplus V_3$ are fixed by $g$.
Hence $g$ fixes $e_2$, which satisfies $Q_2(e_2) = 1$.
If $g = (\tau \sigma^i,\eta)$,
by Lemma \ref{Sublemma:GroupTheoreticLemma} (\ref{Sublemma:GroupTheoreticLemma:Condition5}),
there is a vector $v \in V_1$ satisfying $Q_1(v) = 1$ and $g v = v$.
Hence $G$ satisfies the condition (\ref{Lemma:GroupTheoreticLemma:2}).
\end{proof}

\begin{rem}
\label{Lemma:RemarkSp(6,2):GroupTheoreticLemma}
There are many subgroups
$G \subset \Sp_{2m}(\F_2)$
satisfying the conditions
(\ref{Lemma:GroupTheoreticLemma:1}), (\ref{Lemma:GroupTheoreticLemma:2})
in Lemma \ref{Lemma:GroupTheoreticLemma}.
When $m=3$, a quick search using \texttt{GAP} (version 4.7.5) shows that
there are 1369 subgroups of $\Sp_{6}(\F_2)$, up to conjugacy.
Among them, 411 subgroups, up to conjugacy,
satisfy the conditions
(\ref{Lemma:GroupTheoreticLemma:1}), (\ref{Lemma:GroupTheoreticLemma:2})
in Lemma \ref{Lemma:GroupTheoreticLemma}.
The subgroup $G$ constructed in the proof of Lemma \ref{Lemma:GroupTheoreticLemma}
is a unique subgroup of $\Sp_{6}(\F_2)$ of order 6, up to conjugacy,
satisfying the conditions
(\ref{Lemma:GroupTheoreticLemma:1}), (\ref{Lemma:GroupTheoreticLemma:2})
in Lemma \ref{Lemma:GroupTheoreticLemma}.
\end{rem}

\begin{rem}
\label{Lemma:GroupTheoreticLemma:Falsem=1}
Lemma \ref{Lemma:GroupTheoreticLemma} does not hold for $m=1$.
On the $\F_2$-vector space $\F_2^{\oplus 2}$ with a standard alternating form,
there are three quadratic forms of Arf invariant $0$.
There is a unique quadratic form of Arf invariant $1$.
We denote it by $Q$.
Quadratic forms on $\F_2^{\oplus 2}$ are written as $Q_v$ for $v \in \F_2^{\oplus 2}$.
Quadratic forms of Arf invariant 0 correspond to non-zero vectors in $\F_2^{\oplus 2}$.
By Lemma \ref{LemmaGL(2,2)},
we see that no subgroup
$G \subset \Sp_{2}(\F_2) = \SL_{2}(\F_2) = \GL_{2}(\F_2)$
satisfies the conditions
(\ref{Lemma:GroupTheoreticLemma:1}), (\ref{Lemma:GroupTheoreticLemma:2})
in Lemma \ref{Lemma:GroupTheoreticLemma}.
This explains why
the local-global principle for the existence of symmetric determinantal
representations holds true for cubics
(cf.\ Theorem \ref{MainTheorem:LocalGlobal}
(\ref{MainTheorem:LocalGlobal:Cubic})),
but it does not hold true for quartics.
\end{rem}

\begin{rem}
Lemma \ref{Lemma:GroupTheoreticLemma} holds true for $m=2$.
Using \texttt{GAP} (version 4.7.5), we see that
there are 56 subgroups of
$\Sp_{4}(\F_2) \cong \mathfrak{S}_6$,
up to conjugacy.
Among them, 12 subgroups, up to conjugacy,
satisfy the conditions
(\ref{Lemma:GroupTheoreticLemma:1}), (\ref{Lemma:GroupTheoreticLemma:2})
in Lemma \ref{Lemma:GroupTheoreticLemma}.
However, the case of $m=2$ is not related to the problem of symmetric determinantal representations
because the genus of a smooth plane curve cannot be equal to 2.
\end{rem}

\subsection{A sufficient condition to violate the local-global principle}
\label{Subsection:Criterion}

We give a sufficient condition for a smooth plane curve
to violate the local-global principle
in terms of the associated mod $2$ Galois representation.

Let $K$ be a global field of characteristic different from two,
and $C \subset \PP^2_K$ be a smooth plane curve of degree $n \geq 4$.
The Jacobian variety $\Jac(C)$ is an abelian variety of dimension $(n-1)(n-2)/2$.
Since $\chara K \neq 2$,
the multiplication-by-2 isogeny
$[2] \colon \Jac(C) \longrightarrow \Jac(C)$
is \'etale, and all $\overline{K}$-rational $2$-torsion points on $\Jac(C)$
are defined over $K^{\mathrm{sep}}$.
Hence $\Jac(C)[2](K^{\mathrm{sep}})$
is an $\F_2$-vector space of dimension $(n-1)(n-2)$.
Since the action of $\Gal(K^{\mathrm{sep}}/K)$ on $\Jac(C)[2](K^{\mathrm{sep}})$
preserves the Weil pairing $e_2$, by choosing a symplectic $\F_2$-basis,
we have the associated mod $2$ Galois representation
\[ \rho_{C,2} \colon \Gal(K^{\mathrm{sep}}/K) \ \longrightarrow \ \Sp_{(n-1)(n-2)}(\F_2). \]
We fix an embedding
$\iota_v \colon K^{\mathrm{sep}} \hookrightarrow K_v^{\mathrm{sep}}$
for each place $v$ of $K$.
We consider
$\Gal(K_v^{\mathrm{sep}}/K_v)$ as a closed subgroup of $\Gal(K^{\mathrm{sep}}/K)$.
The embedding
\[ \Gal(K_v^{\mathrm{sep}}/K_v) \hookrightarrow \Gal(K^{\mathrm{sep}}/K) \]
is unique up to conjugation.

\begin{prop}
\label{Proposition:CounterexampleSufficientCondition}
Assume that {\em at least one} of the following conditions is satisfied:
\begin{enumerate}
\renewcommand{\labelenumi}{(\alph{enumi})}
\renewcommand{\theenumi}{{\rm\alph{enumi}}}
\item \label{Proposition:CounterexampleSufficientCondition:ExistenceLocalRationalPoints}
  $C$ has a $K_v$-rational point for each place $v$ of $K$, or

\item \label{Proposition:CounterexampleSufficientCondition:OddDegree}
  the degree $n$ is odd.
\end{enumerate}
Moreover, assume that {\em all} of the following conditions are satisfied:
\begin{enumerate}
\setcounter{enumi}{2}
\renewcommand{\labelenumi}{(\alph{enumi})}
\renewcommand{\theenumi}{{\rm\alph{enumi}}}
\item \label{Proposition:CounterexampleSufficientCondition:GaloisImageCondition}
  the image of $\rho_{C,2}$
  satisfies the conditions
  (\ref{Lemma:GroupTheoreticLemma:1}), (\ref{Lemma:GroupTheoreticLemma:2})
  in Lemma \ref{Lemma:GroupTheoreticLemma},

\item \label{Proposition:CounterexampleSufficientCondition:GaloisImageLocalCyclicity}
  the image $\rho_{C,2}\big( \Gal(K_v^{\mathrm{sep}}/K_v) \big)$ is a cyclic group
  for each place $v$ of $K$, and

\item \label{Proposition:CounterexampleSufficientCondition:EvenThetaNonEffective}
  all even theta characteristics on $C_{K^{\mathrm{sep}}}$ are non-effective.
\end{enumerate}
Then $C$ admits a symmetric determinantal representation over $K_v$ for each place $v$ of $K$,
and $C$ does not admit a symmetric determinantal representation over $K$.
\end{prop}

\begin{proof}
Let $v$ be a place of $K$.
Note that
\[ \Jac(C)[2](K^{\mathrm{sep}}) = \Jac(C)[2](K_v^{\mathrm{sep}}) \]
because $\Jac(C)[2]$ is a finite \'etale group scheme over $K$.

By the conditions
(\ref{Proposition:CounterexampleSufficientCondition:GaloisImageCondition}),
(\ref{Proposition:CounterexampleSufficientCondition:GaloisImageLocalCyclicity}),
there is a quadratic form on $\Jac(C)[2](K_v^{\mathrm{sep}})$ with Arf invariant 0
fixed by $\Gal(K_v^{\mathrm{sep}}/K_v)$
(see the condition (\ref{Lemma:GroupTheoreticLemma:2})
in Lemma \ref{Lemma:GroupTheoreticLemma}).
Proposition \ref{Proposition:MumfordBijection} applied to $C_{K_v}$
shows that there is an even theta characteristic $\mathcal{L}_{K_v^{\mathrm{sep}}}$
on $C_{K_v^{\mathrm{sep}}}$ such that
\[ [\mathcal{L}_{K_v^{\mathrm{sep}}}] \in \Pic_{C/K}(K_v^{\mathrm{sep}}) \]
is fixed by $\Gal(K_v^{\mathrm{sep}}/K_v)$.

If the condition
(\ref{Proposition:CounterexampleSufficientCondition:ExistenceLocalRationalPoints})
is satisfied,
$C_{K_v}$ satisfies the condition
(\ref{Corollary:MumfordBijectionGaloisEquivariance:ExistenceRationalPoint})
of Corollary \ref{Corollary:MumfordBijectionGaloisEquivariance}.
If the condition
(\ref{Proposition:CounterexampleSufficientCondition:OddDegree})
is satisfied,
there is a finite extension $M_v/K_v$ of odd degree such that
$C$ has an $M_v$-rational point.
Hence $C$ satisfies the condition
(\ref{Corollary:MumfordBijectionGaloisEquivariance:ExistenceRationalPointOddDegree})
of Corollary \ref{Corollary:MumfordBijectionGaloisEquivariance}.
In both cases,
$\mathcal{L}_{K_v^{\mathrm{sep}}}$ comes from an even theta characteristic on $C_{K_v}$
by Corollary \ref{Corollary:MumfordBijectionGaloisEquivariance}.
It is non-effective by the condition
(\ref{Proposition:CounterexampleSufficientCondition:EvenThetaNonEffective}).
Therefore, $C$ admits a symmetric determinantal representation over $K_v$
by Theorem \ref{ExistenceCriterion}.

Finally, by the condition
(\ref{Proposition:CounterexampleSufficientCondition:GaloisImageCondition}),
there does not exist
a quadratic form on $\Jac(C)[2](K^{\mathrm{sep}})$ with Arf invariant 0
fixed by $\Gal(K^{\mathrm{sep}}/K)$
(see the condition (\ref{Lemma:GroupTheoreticLemma:1})
in Lemma \ref{Lemma:GroupTheoreticLemma}).
Hence there does not exist an even theta characteristic on $C$.
By Theorem \ref{ExistenceCriterion},
$C$ does not admit a symmetric determinantal representation over $K$.
\end{proof}

\begin{rem}
\label{Proposition:Remark:CounterexampleSufficientCondition:n=4}
When $n=4$, the condition 
(\ref{Proposition:CounterexampleSufficientCondition:EvenThetaNonEffective})
is always satisfied because even theta characteristics on
smooth plane quartics are non-effective.
(See Remark \ref{Remark:EvenThetaCharacteristicQuartic}.)
\end{rem}

\subsection{Counterexamples to the local-global principle for quartics}
\label{Subsection:Counterexample}

\begin{proof}[of Theorem \ref{MainTheorem2}]
Recall that $K$ is a global field of characteristic different from two,
and $C \subset \PP^2_{K}$ a smooth plane quartic over $K$
such that the associated mod 2 Galois representation
\[ \rho_{C,2} \colon \Gal(K^{\mathrm{sep}}/K) \longrightarrow \Sp_{6}(\F_2) \]
is surjective.

We first claim that there is a finite separable extension $K'/K$ such that
$C$ has a $K'$-rational point and the restriction of
$\rho_{C,2}$ to $\Gal(K^{\mathrm{sep}}/K')$ is still surjective.
In some examples (e.g.,\ Remark \ref{Counterexample:Remark1}),
we can find a $K$-rational point explicitly,
and we can simply take $K = K'$.
But $C$ might not have a $K$-rational point in general.

There is a line $\ell \subset \PP^2_{K}$ such that
$\ell \cap C$ is smooth over $K$ by Bertini's theorem
(\cite[Corollaire 6.11 (2)]{Jouanolou}).
Since $\ell \cap C$ is defined by a separable polynomial of degree 4,
there is a separable extension $K'/K$ of degree $\leq 4$
such that $C$ has a $K'$-rational point.
Since 
$\Gal(K^{\mathrm{sep}}/K')$ is a closed subgroup of $\Gal(K^{\mathrm{sep}}/K)$
of index $\leq 4$, its image
\[ \rho_{C,2}\big( \Gal(K^{\mathrm{sep}}/K') \big) \subset \Sp_{6}(\F_2) \]
is a subgroup of index $\leq 4$.
Since $\Sp_{6}(\F_2)$ has no proper subgroup of index $\leq 4$,
the restriction of $\rho_{C,2}$ to $\Gal(K^{\mathrm{sep}}/K')$ is surjective.
(A proper subgroup of $\Sp_{6}(\F_2)$ of smallest index is
conjugate to the orthogonal group of a quadratic form with Arf invariant 1,
which has index 28.)
Replacing $K$ by $K'$, we may assume that $C$ has a $K$-rational point.

Next, we shall replace $K$ by a finite extension of it as follows.
We choose a subgroup
\[ G \subset \rho_{C,2}\big( \Gal(K^{\mathrm{sep}}/K) \big) = \Sp_{6}(\F_2) \]
satisfying the conditions in Lemma \ref{Lemma:GroupTheoreticLemma}.
There are 411 subgroups of $\Sp_{6}(\F_2)$ with this property, up to conjugacy
(Remark \ref{Lemma:RemarkSp(6,2):GroupTheoreticLemma}).
We may choose any one of them.
We take a finite separable extension $M/K$ such that
$\Gal(K^{\mathrm{sep}}/M) = \rho_{C,2}^{-1}(G)$.

By Remark \ref{Proposition:Remark:CounterexampleSufficientCondition:n=4},
the smooth plane quartic $C_M$ over $M$ satisfies the conditions
(\ref{Proposition:CounterexampleSufficientCondition:ExistenceLocalRationalPoints}),
(\ref{Proposition:CounterexampleSufficientCondition:GaloisImageCondition}),
(\ref{Proposition:CounterexampleSufficientCondition:EvenThetaNonEffective})
of Proposition \ref{Proposition:CounterexampleSufficientCondition},
but it might not satisfy the condition
(\ref{Proposition:CounterexampleSufficientCondition:GaloisImageLocalCyclicity}).
Note that,
if $v$ is an archimedean place or a place where $\rho_{C,2}$ is unramified,
the image
\[ \rho_{C,2}\big( \Gal(M_v^{\mathrm{sep}}/M_v) \big) \subset \Sp_{6}(\F_2) \]
is a cyclic group generated
by the image of the complex conjugation or the geometric Frobenius element,
and the condition
(\ref{Proposition:CounterexampleSufficientCondition:GaloisImageLocalCyclicity})
of Proposition \ref{Proposition:CounterexampleSufficientCondition}
is satisfied at $v$.
Hence the number of places $v$ where 
the condition
(\ref{Proposition:CounterexampleSufficientCondition:GaloisImageLocalCyclicity})
of Proposition \ref{Proposition:CounterexampleSufficientCondition}
is not satisfied is finite.
We denote them by $v_1,\ldots,v_r$.
Since $\rho_{C,2}\big( \Gal(M_{v_i}^{\mathrm{sep}}/M_{v_i}) \big)$ is a finite non-cyclic group,
there is a finite extension $M'_{v_i}$ of $M_{v_i}$ such that
$\rho_{C,2}\big( \Gal(M_{v_i}^{\mathrm{sep}}/M'_{v_i}) \big)$ is cyclic.
Take a finite separable extension $M'/K$ such that
$M'/K$ is linearly disjoint from
$(K^{\mathrm{sep}})^{\Ker \rho_{C,2}}$
and $(M' M)_{w_i}$ contains all conjugates of $M'_{v_i}$ for all $i$,
where $w_i$ is a place of the composite $M' M$ above $v_i$.
The existence of such $M'/K$ is a consequence of Krasner's lemma
and the weak approximation theorem.

We put $L := M' M$.
The smooth plane quartic $C_L$ over $L$ satisfies
the conditions
(\ref{Proposition:CounterexampleSufficientCondition:ExistenceLocalRationalPoints}),
(\ref{Proposition:CounterexampleSufficientCondition:GaloisImageCondition}),
(\ref{Proposition:CounterexampleSufficientCondition:GaloisImageLocalCyclicity}),
(\ref{Proposition:CounterexampleSufficientCondition:EvenThetaNonEffective})
of Proposition \ref{Proposition:CounterexampleSufficientCondition}.
Hence $C_L$ is a desired counterexample to the local-global principle
by Proposition \ref{Proposition:CounterexampleSufficientCondition}.

The proof of Theorem \ref{MainTheorem2} is complete.
\end{proof}

The above proof shows that
the assumption on the surjectivity of the associated mod 2 Galois representation
can be weakened if $C$ has a $K$-rational point.

\begin{cor}
\label{Corollary:MainTheorem2:RationalPoint}
Let $C \subset \PP^2_K$ be a smooth plane quartic over a global field $K$
of characteristic different from two
satisfying both of the following conditions:
\begin{enumerate}
\renewcommand{\labelenumi}{(\alph{enumi})}
\renewcommand{\theenumi}{{\rm\alph{enumi}}}
\item $C$ has a $K$-rational point, and
\item the image of the associated mod 2 Galois representation $\rho_{C,2}$ contains
a subgroup of $\Sp_{6}(\F_2)$
satisfying the conditions in Lemma \ref{Lemma:GroupTheoreticLemma}.
\end{enumerate}
Then there is a finite separable extension $L/K$
such that $C$ admits a symmetric determinantal representation over $L_w$ for each place $w$ of $L$,
and $C$ does not admit a symmetric determinantal representation over $L$.
\end{cor}

We shall prove Corollary \ref{Corollary:MainTheorem2}
by combining Theorem \ref{MainTheorem2} and the results of Harris, and Shioda.

\begin{proof}[of Corollary \ref{Corollary:MainTheorem2}]
It is known that the associated mod 2 Galois representation of the generic family
of plane quartics over a prime field of characteristic different from
$2,3,5,7,11,29,1229$
is surjective (\cite[p.\ 721]{HarrisGalois}, \cite[Theorem 7]{Shioda}).
Since global fields are Hilbertian (\cite[Theorem 13.3.5]{FriedJarden}),
applying Hilbert's irreducibility theorem, we obtain
infinitely many smooth plane quartics over $\Q$ or $\F_p(T)$ (for $p \neq 2,3,5,7,11,29,1229$)
whose associated mod $2$ Galois representations are surjective.
Hence the assertion of Corollary \ref{Corollary:MainTheorem2} follows from Theorem \ref{MainTheorem2}.
\end{proof}

\begin{rem}
The condition ``$p \neq 2,3,5,7,11,29,1229$''
comes from the condition in Shioda's paper (\cite[Theorem 7]{Shioda}).
In \cite[p.\ 68]{Shioda},
it is alluded that the restriction on the characteristics can be relaxed.
If \cite[Theorem 7]{Shioda} holds in characteristic $p \in \{ 3,5,7,11,29,1229 \}$,
Corollary \ref{Corollary:MainTheorem2} also holds in that characteristic.
The case of characteristic two is completely different (\cite{IshitsukaItoCharacteristic2}).
\end{rem}

\subsection{Concluding remarks}

\begin{rem}
\label{Counterexample:Remark1}
It is possible to take explicit smooth plane quartics
satisfying the conditions of Theorem \ref{MainTheorem2}.
When $K = \Q$, we may take $C$ to be a smooth plane quartic defined by the equation
\[
X_0 X_2^3 + X_2 (X_0^3 + X_0^2 X_1 + X_1^3) + X_0^4 + X_0^3 X_1 + X_0^2 X_1^2 + X_1^4 = 0,
\]
or the equation
\[
X_0^2 X_1^2 - X_0 X_1^3 - X_0^3 X_2 - 2 X_0^2 X_2^2 + X_1^2 X_2^2 - X_0 X_2^3 + X_1 X_2^3 = 0.
\]
The first equation is taken from \cite{Shioda},
and the second one is taken from \cite{BruinPoonenStoll}, \cite{Bruin}.
Note that the quartics defined by the above equations have $\Q$-rational points
such as $(0,0,1)$.
\end{rem}

\begin{rem}
\label{Counterexample:Remark2}
In principle,
the extension $L/K$ in the proof of Theorem \ref{MainTheorem2} can be
taken explicitly if we can calculate the associated mod 2 Galois representations.
The required condition on $M'/K$ can be seen from the ramification of $\rho_{C,2}$.
For example, let us consider the smooth plane quartic $C$ over $\Q$ defined by
the following equation (\cite{BruinPoonenStoll}, \cite{Bruin}):
\[
X_0^2 X_1^2 - X_0 X_1^3 - X_0^3 X_2 - 2 X_0^2 X_2^2 + X_1^2 X_2^2 - X_0 X_2^3 + X_1 X_2^3 = 0.
\]
The extension
$\overline{\Q}^{\Ker \rho_{C,2}}/\Q$
is a Galois extension with Galois group
$\Sp_6(\F_2)$ of degree $1451520$ ramified only at $2,41,347$
(\cite[12.9.2]{BruinPoonenStoll}).
If we choose $G \subset \Sp_{6}(\F_2)$ to be the subgroup of order 6
constructed in the proof of Lemma \ref{Lemma:GroupTheoreticLemma},
$M/\Q$ is a subextension of $\overline{\Q}^{\Ker \rho_{C,2}}/\Q$ of
degree $241920$ with
$\Gal\big( \overline{\Q}^{\Ker \rho_{C,2}}/M \big) = G$.
If we take an extension $M'/\Q$ such that
$M'$ is linearly disjoint from $\overline{\Q}^{\Ker \rho_{C,2}}$ and
$\big( M' \, \overline{\Q}^{\Ker \rho_{C,2}} \big)/(M' M)$
is unramified at all places above $2,41,347$,
the composite $L = M' M$ satisfies all the required conditions,
and $C_L$ is a counterexample to the local-global principle.
It is possible to reduce the degree $[M : \Q]$ if you prefer.
According to calculations done by \texttt{GAP} (version 4.7.5),
the largest subgroup of $\Sp_{6}(\F_2)$ satisfying
the conditions in Lemma \ref{Lemma:GroupTheoreticLemma} is of order $1440$.
(There are exactly two subgroups of $\Sp_{6}(\F_2)$ of order $1440$, up to conjugacy.
Both are isomorphic to $\Z/2\Z \times \mathfrak{S}_6$.
Only one of them has the required properties.)
If we take this group as $G$, 
we reduce the extension degree $[M : \Q]$ to $1008$.
\end{rem}

\begin{rem}
\label{Counterexample:Remark3}
The degree $[L : K]$ of the extension $L/K$ in the proof of Theorem \ref{MainTheorem2} is rather large.
It is an interesting problem to construct counterexamples to the local-global principle
over smaller global fields such as $\Q$ or $\F_p(T)$ for $p \neq 2$.
We expect that counterexamples to the local-global principle exist
over $\Q$ or $\F_p(T)$ for $p \neq 2$
in any degree $n \geq 4$ because our group theoretic lemma
(Lemma \ref{Lemma:GroupTheoreticLemma}) holds for any $m \geq 3$
and Proposition \ref{Proposition:CounterexampleSufficientCondition} holds for any $n \geq 4$.
Finding counterexamples to the local-global principle in higher degree
seems a challenging computational problem
because an algorithm to calculate the associated mod 2 Galois representation of
smooth plane curves of higher degree is yet to be developed.
(See \cite{BruinPoonenStoll}, \cite{Bruin} for recent results on
explicit calculations for smooth plane quartics.)
\end{rem}

\subsection*{Acknowledgements}

The work of the first author was supported by JSPS KAKENHI Grant Number 13J01450.
The work of the second author 
was supported by JSPS KAKENHI Grant Number 20674001 and 26800013.
Some calculations done by \texttt{GAP} (version 4.7.5)
on subgroups of the finite symplectic group $\Sp_{2m}(\F_2)$
were very helpful
when we studied the action of Galois groups on theta characteristics.
The authors would like to thank an anonymous referee
for their useful comments and suggestions.


\begin{thebibliography}{ZZZ}
\bibitem{BeauvillePrym}
  Beauville, A.,
  \textit{Vari\'et\'es de Prym et jacobiennes interm\'ediaires},
  Ann.\ Sci.\ \'Ecole Norm.\ Sup.\ (4) 10 (1977), no.\ 3, 309-391.
\bibitem{BeauvilleDeterminantal}
  Beauville, A., \textit{Determinantal hypersurfaces},
  Michigan Math.\ J.\ 48 (2000), 39-64.
\bibitem{BhargavaGrossWang}
  Bhargava, M., Gross, B.\ H., Wang, X.,
  \textit{Arithmetic invariant theory II},
  to appear in Progress in Mathematics, Representations of Lie Groups,
  in Honor of David A. Vogan, Jr.\ on his 60th birthday,
  \texttt{arXiv:1310.7689}
\bibitem{BoschLuetkebohmertRaynaud}
  Bosch, S., L\"utkebohmert, W., Raynaud, M., \textit{N\'eron models},
  Ergebnisse der Mathematik und ihrer Grenzgebiete (3), 21. Springer-Verlag, Berlin, 1990.
\bibitem{Bruin}
  Bruin, N., \textit{Success and challenges in determining the rational points on curves},
  Proceedings of the Tenth Algorithmic Number Theory Symposium, The Open Book Series 1-1 (2013), 187-212.
\bibitem{BruinPoonenStoll}
  Bruin, N., Poonen, B., Stoll, M.,
  \textit{Generalized explicit descent and its application to curves of genus 3},
  preprint, \texttt{arXiv:1205.4456}
\bibitem{Catanese}
  Catanese, F.,
  \textit{Babbage's conjecture, contact of surfaces, symmetric determinantal varieties and applications},
  Invent.\ Math.\ 63 (1981), no.\ 3, 433-465.
\bibitem{CookThomas}
  Cook, R.\ J., Thomas, A.\ D., \textit{Line bundles and homogeneous matrices},
  Quart.\ J.\ Math.\ Oxford Ser.\ (2) 30 (1979), no.\ 120, 423-429.
\bibitem{Dixon}
  Dixon, A.\ C., \textit{Note on the reduction of a ternary quantic to a symmetric determinant},
  Proc.\ Cambridge Phil.\ Soc.\ 11 (1902), 350-351.
\bibitem{DixonMortimer}
  Dixon, J.\ D., Mortimer, B., \textit{Permutation groups},
  Graduate Texts in Mathematics, 163. Springer-Verlag, New York, 1996.
\bibitem{Dye}
  Dye, R.\ H.,
  \textit{Interrelations of symplectic and orthogonal groups in characteristic two},
  J.\ Algebra 59 (1979), no.\ 1, 202-221.
\bibitem{Dolgachev}
  Dolgachev, I.\ V., \textit{Classical algebraic geometry - A modern view},
  Cambridge University Press, Cambridge, 2012.
\bibitem{Edge}
  Edge, W.\ L., \textit{Determinantal representations of $x^4 + y^4 + z^4$},
  Math.\ Proc.\ Cambridge Phil.\ Soc.\ 34 (1938), 6-21.
\bibitem{FriedJarden}
  Fried, M.\ D., Jarden, M., \textit{Field arithmetic},
  Third edition, Ergebnisse der Mathematik und ihrer Grenzgebiete.\ 3.\ Folge.\ A Series of Modern Surveys in Mathematics,
  11.\ Springer-Verlag, Berlin, 2008.
\bibitem{GrossHarris}
  Gross, B.\ H., Harris, J.,
  \textit{On some geometric constructions related to theta characteristics},
  Contributions to automorphic forms, geometry, and number theory, 279-311,
  Johns Hopkins Univ. Press, Baltimore, MD, 2004.
\bibitem{GrothendieckBrauerIII}
  Grothendieck, A., \textit{Le groupe de Brauer III},
  Dix Expos\'es sur la Cohomologie des Sch\'emas,
  88-188 North-Holland, Amsterdam; Masson, Paris, 1968.
\bibitem{HarrisGalois}
  Harris, J., \textit{Galois groups of enumerative problems}, Duke Math.\ J.\ 46 (1979), no.\ 4, 685-724.
\bibitem{Hartshorne}
  Hartshorne, R., \textit{Algebraic geometry}, Graduate Texts in Mathematics, No.\ 52.,
  Springer-Verlag, New York-Heidelberg, 1977.
\bibitem{Hesse1}
  Hesse, O., \textit{\"Uber Elimination der Variabeln aus drei algebraischen
  Gleichungen von zweiten Graden, mit zwei Variabeln},
  J.\ Reine Angew.\ Math.\ 28 (1844) 68-96.
\bibitem{Hesse2}
  Hesse, O., \textit{\"Uber die Doppeltangenten der Curven vierter Ordnung},
  J.\ Reine Angew.\ Math.\ 49 (1855) 279-332.
\bibitem{Ho}
  Ho, W., \textit{Orbit parametrizations of curves},
  Ph.D\ Thesis, Princeton University, 2009.
\bibitem{Ishitsuka}
  Ishitsuka, Y., \textit{Orbit parametrizations of theta characteristics on hypersurfaces over arbitrary fields},
  preprint, \texttt{arXiv:1412.6978}
\bibitem{IshitsukaItoCharacteristic2}
  Ishitsuka, Y., Ito, T., \textit{The local-global principle for symmetric determinantal representations of smooth plane curves in characteristic two}, preprint, \texttt{arXiv:1412.8343}
\bibitem{IshitsukaItoFermat}
  Ishitsuka, Y., Ito, T., \textit{On the symmetric determinantal representations of the Fermat curves of prime degree}, preprint, \texttt{arXiv:1412.8345},
to appear in International Journal of Number Theory.
\bibitem{Jouanolou}
  Jouanolou, J.-P., \textit{Th\'eor\`emes de Bertini et applications},
  Progress in Mathematics, 42.\ Birkh\"auser Boston, Inc., Boston, MA, 1983.
\bibitem{KleimanPicard}
  Kleiman, S.\ L., \textit{The Picard scheme}, Fundamental algebraic geometry, 235-321,
  Math.\ Surveys Monogr., 123, Amer.\ Math.\ Soc., Providence, RI, 2005.
\bibitem{Liu}
  Liu, Q., \textit{Algebraic geometry and arithmetic curves},
  Translated from the French by Reinie Ern\'e.\ Oxford Graduate Texts in Mathematics,
  6.\ Oxford Science Publications.\ Oxford University Press, Oxford, 2002.
\bibitem{Meyer-Brandis}
  Meyer-Brandis, T.,
  \textit{Ber\"uhrungssysteme und symmetrische Darstellungen ebener Kurven},
  Diplomarbeit, Johannes Gutenberg-Universit\"at Mainz, 1998.
\bibitem{MumfordTheta}
  Mumford, D., \textit{Theta characteristics of an algebraic curve},
  Ann.\ Sci.\ \'Ecole Norm.\ Sup.\ (4) 4 (1971), 181-192.
\bibitem{NeukirchSchmidtWingberg}
  Neukirch, J., Schmidt, A., Wingberg, K., \textit{Cohomology of number fields},
  Second edition.\ Grundlehren der Mathematischen Wissenschaften, 323.\ Springer-Verlag, Berlin, 2008.
\bibitem{Room}
  Room, T.\ G., \textit{The geometry of determinantal loci}, Cambridge Univ.\ Press, Cambridge, 1938.
\bibitem{SerreAbelianl-adic}
  Serre, J.-P., \textit{Abelian $\ell$-adic representations and elliptic curves},
  McGill University lecture notes written with the collaboration of Willem Kuyk and John Labute W. A. Benjamin, Inc.,
  New York-Amsterdam, 1968.
\bibitem{SerreLocalFields}
  Serre, J.-P., \textit{Local fields}, Graduate Texts in Mathematics, 67.\ Springer-Verlag, New York-Berlin, 1979.
\bibitem{Shioda}
  Shioda, T., \textit{Plane quartics and Mordell-Weil lattices of type $E_7$},
  Comment.\ Math.\ Univ.\ St.\ Paul.\ 42 (1993), no.\ 1, 61-79.
\bibitem{TateGCFT}
  Tate, J.\ T., \textit{Global class field theory}, in Algebraic Number Theory (J.\ W.\ S.\ Cassels and A.\ Fr\"ohlich eds.),
  Proc.\ Instructional Conf., Brighton, 1965, 162-203.
\bibitem{Tyurin}
  Tyurin, A.\ N., \textit{On intersection of quadrics}, Russian Math.\ Surveys 30 (1975), 51-105.
\bibitem{Vinnikov1}
  Vinnikov, V., \textit{Complete description of determinantal representations of smooth irreducible curves},
  Linear Algebra Appl.\ 125 (1989), 103-140.
\bibitem{Vinnikov2}
  Vinnikov, V., \textit{Selfadjoint determinantal representations of real plane curves},
  Math.\ Ann.\ 296 (1993), no.\ 3, 453-479.
\bibitem{Wall}
  Wall, C.\ T.\ C.,
  \textit{Nets of quadrics and theta-characteristics of singular curves},
  Philos.\ Trans.\ Roy.\ Soc.\ London, Ser.\ A 289 (1978) 229-269.
\end{thebibliography}
\end{document}